\documentclass[12pt,reqno]{amsart}
\usepackage{amssymb}
\usepackage{amsthm}
\usepackage{amsmath}
\usepackage{latexsym}

\newtheorem{theorem}{Theorem}[section]
\newtheorem{lemma}[theorem]{Lemma}
\newtheorem{proposition}[theorem]{Proposition}
\newtheorem{corollary}[theorem]{Corollary}

\newtheorem{question}[theorem]{Question}

\theoremstyle{definition}
\newtheorem{definition}[theorem]{Definition}

\newtheorem{remark}[theorem]{Remark}

\numberwithin{equation}{section}

\newcounter{smallromans}

  {\end{list}}

\newcounter{smallalphs}

  {\end{list}}

\renewcommand{\leq}{\ensuremath{\leqslant}}
\renewcommand{\le}{\ensuremath{\leqslant}}
\renewcommand{\geq}{\ensuremath{\geqslant}}
\renewcommand{\ge}{\ensuremath{\geqslant}}

\renewcommand{\epsilon}{\ensuremath{\varepsilon}}

\def\Ave{\mathop{\text {Ave}}}

\begin{document}
\title{Weight-almost greedy bases}

\author[S. Dilworth]{S. J. Dilworth}
\address{Department of Mathematics, University of South Carolina, Columbia, SC, USA}
\email{dilworth@math.sc.edu}

\author[D. Kutzarova]{Denka Kutzarova}
\address{Department of Mathematics, University of Illinois at Urbana--Champaign, Urbana, IL, USA;  Institute of Mathematics and Informatics, Bulgarian Academy of Sciences}
\email{denka@math.uiuc.edu}

\author[V. Temlyakov]{Vladimir Temlyakov}
\address{Department of Mathematics, University of South Carolina, Columbia, SC, USA; Steklov Institute of Mathematics, Moscow, Russia; Lomonosov Moscow State University, Moscow, Russia}
\email{temlyak@math.sc.edu}

\author[B. Wallis]{Ben Wallis}
\address{Department of Mathematical Sciences, Northern Illinois University, Dekalb, IL, USA}
\email{benwallis@live.com}

\begin{abstract}
We introduce the notion of a \textit{weight-almost greedy} basis and show that a basis for a real Banach space is $w$-almost greedy if and only if it is both quasi-greedy and $w$-democratic. We also introduce the notion of \textit{weight-semi-greedy} basis and show that
a $w$-almost greedy basis is  $w$-semi-greedy and that the converse holds if the Banach space has finite cotype.
\end{abstract}
\thanks{The first author was supported by the National Science Foundation under Grant Number DMS--1361461; the third author was supported  by the Russian Federation Government Grant No. 14.W03.31.0031. The first and second authors were supported by the Workshop in Analysis and Probability at Texas A\&M University in 2017.}

\maketitle

\section{Introduction}

Let $(e_n)_{n=1}^\infty$ be a normalized basis for a Banach space $X$ with biorthogonal functionals $(e_n^*)_{n=1}^\infty$.  Konyagin and Temlyakov defined  \cite{KT} the {\bf thresholding greedy algorithm (TGA)} for $(e_n)_{n=1}^\infty$ as the sequence $(G_m)_{m=1}^\infty$ of functions $G_m:X\to X$ by letting $G_m(x)$ denote the vector obtained by taking the $m$ largest coefficients of $x=\sum e_n^*(x)e_n\in X$. The TGA provides a theoretical model for the thresholding procedure that is used  in image compression.

They defined  the basis $(e_n)$ to be
{\it greedy}
if the TGA is optimal in the sense that $G_n(x)$ is
 essentially the best $n$-term approximation to $x$ using the basis vectors, i.e.\
if  there exists a constant
$K$ such that for all $x \in X$ and $m \in \mathbb{N}$, we have
\begin{equation} \|x-G_m(x)\|\le K\inf\{\|x-\sum_{n\in
A}\alpha_ne_n\|:\
|A|=m,\
\alpha_n\in\mathbb R,\ n\in A\}.\end{equation}
They then showed
 that greedy bases can be simply characterized as
unconditional bases with the additional property of being {\it
democratic}, i.e. for some $\Delta>0$, we have $$\|\sum_{n\in A}e_n\|\le
\Delta \|\sum_{n\in B}e_n\|\quad \text{whenever $|A|\le |B|$}.$$
They also defined a basis to be \textit{quasi-greedy} if  there exists a constant $K$ 
such that $\|G_m(x)\| \le K\|x\|$
for all $x \in X$ and $m \in \mathbb{N}$.
Subsequently,
Wojtaszczyk \cite{W} proved that these are precisely the bases
for which the TGA merely converges, i.e.
$\lim_{m\to\infty}G_m(x)=x$ for $x\in X$.

The class of \textit{almost greedy} bases was introduced in \cite{DKKT}.  Let us
denote the biorthogonal sequence by $(e_n^*)$. Then
$(e_n)$ is almost greedy if there is a constant $K$ such that
\begin{equation} \|x-G_m(x)\|\le K\inf\{\|x-\sum_{n\in
A}e_n^*(x)e_n\|:\
|A|=m
\}\qquad x\in X, \ m\in\mathbb
N.\end{equation} 
It was proved in \cite{DKKT} 
 that $(e_n)$ is almost greedy if and only if $(e_n)$ is quasi-greedy and democratic.

The class of \textit{semi-greedy} bases was introduced in \cite{DKK}. A basis is semi-greedy if there is a constant $C$ such that
for all $x \in X$ and $m \in \mathbb{N}$

\begin{equation} \|x-\overline{G}_m(x)\|\le C\inf\{\|x-\sum_{n\in
A}\alpha_ne_n\|:\
|A|=m,\
\alpha_n\in\mathbb R,\ n\in A\},\end{equation} 
where $\overline{G}_m(x)$ is the best  $m$- term approximation to $x$ with the same (or possibly smaller) support as $G_m(x)$ (i.e., $\overline{G}_m(x)$ is supported on the basis vectors corresponding to the $m$ largest coefficients of $x$). It was proved in \cite{DKK} that an almost greedy basis is semi-greedy and that the converse holds if $X$ has finite cotype.
 
Later, Kerkyacharian, Picard and Temlyakov \cite{KPT} (see also \cite{Tbook}, Section 1.3) defined the notion of greedy basis with respect to a weight function, \textit{weight-greedy} basis, and proved a criterion similar to the one for greedy bases.
Their generalization was inspired by research of Cohen, DeVore and Hochmuth \cite{CDH}. Recently, Berna and Blasco \cite{BB} showed that greedy bases can be characterized as those where the error term using $m$-greedy approximant is uniformly bounded by the best $m$-term approximation with respect to polynomials with constant coefficients in the context of the weak greedy algorithm with weights.

In this paper we introduce the notion of a \textit{weight-almost greedy} basis and show that a basis for a real Banach space is $w$-almost greedy if and only if it is both quasi-greedy and $w$-democratic.

From here onward, a {\bf weight} $w=(w_i)_{i=1}^\infty$ will always be a sequence of positive real numbers.  If $A\subset\mathbb{N}$ then we let $w(A)=\sum_{i\in A}w_i$ denote the {\bf $\boldsymbol{w}$-measure} of $A$.

Let us give a more precise definition of each $G_m(x)$.  If $(a_i)_{i=1}^\infty\in c_0$ then we denote by $(a_i^*)_{i=1}^\infty$ the nonincreasing rearrangement of $(|a_i|)_{i=1}^\infty$.  Let $\tau$ be any permutation of $\mathbb{N}$ satisfying $|a_{\tau(i)}|=a_i^*$ for all $i\in\mathbb{N}$.  Then $G_{m,\tau}(x)=\sum_{i=1}^mx^*_{\tau(i)}(x)x_{\tau(i)}$. 
We next set 
$$\Lambda_{m,\tau,x}:=\{\tau(1),\cdots,\tau(m)\},$$ the set of indices of the coefficients associated with $G_{m,\tau}(x)$, called the {\bf support} of $G_{m,\tau}(x)$.  When $\tau$ is understood (or irrelevant) we write $G_m=G_{m,\tau}$  and $\Lambda_m=\Lambda_{m,\tau,x}$,  and when $m$ and $x$ are also understood we write $\Lambda=\Lambda_{m,\tau,x}$.

As is usual, we write
$$\mathbb{N}^m=\{A\subset\mathbb{N}: |A|=m\}\;\;\;\text{ and }\;\;\;\mathbb{N}^{<\infty}=\bigcup_{m=0}^\infty\mathbb{N}^m.$$
Throughout, $C$ and $D$ will denote constants in $[1,\infty)$.

\begin{definition}We say that a normalized basis $(e_n)_{n=1}^\infty$ for a Banach space $X$ is {\bf $\boldsymbol{w}$-almost greedy with constant $\boldsymbol{K}$} whenever
$$\|x-G_m(x)\|\leq K\inf\left\{\|x-\sum_{n\in A}e_n^*(x)e_n\|:A\in\mathbb{N}^{<\infty},\;w(A)\leq w(\Lambda)\right\}$$
for all $m\in\mathbb{N}$ and $x\in X$, and independent of choice of $\tau$.
\end{definition}

\begin{remark}\label{x-bound}In the above definition we can always take $A=\emptyset$ to obtain
$$\inf\left\{\|x-\sum_{n\in A}e_n^*(x)e_n\|:A\in\mathbb{N}^{<\infty},\;w(A)\leq w(\Lambda)\right\}\leq\|x\|.$$\end{remark}

\begin{definition}A normalized basis $(e_n)_{n=1}^\infty$ is {\bf $\boldsymbol{w}$-democratic with constant $\boldsymbol{D}$} whenever $w(A)\leq w(B)$ for $A,B\in\mathbb{N}^{<\infty}$ implies
$$\|\sum_{n\in A}e_n\|\leq D\|\sum_{n\in B}e_n\|.$$
\end{definition}

\begin{remark}\label{subsequence}Observe that a subsequence $(x_{n_k})_{k=1}^\infty$ of a $D$-$w$-democratic sequence $(e_n)_{n=1}^\infty$ is $D$-$w'$-democratic, where $w'=(w_{n_k})_{k=1}^\infty$ is the corresponding subsequence of $w$.  Indeed, if $A\in\mathbb{N}^{<\infty}$ then we write $A'=\{n_a:a\in A\}$ so that $w'(A)=w(A')$.\end{remark}

When $w=(1,1,1,\cdots)$, we have $w(A)=|A|$ for each $A\in\mathbb{N}^{<\infty}$.  In this case, a $w$-almost greedy (resp., $w$-democratic) basis is called, simply, {\bf almost greedy} (resp., {\bf democratic}).

Finally, let us remark that we will often be forced to work with real Banach spaces, as most of the proofs are invalid in the complex setting.

\section{$w$-almost greedy bases are quasi-greedy and $w$-democratic}

\begin{lemma}\label{quasi-greedy}Every $K$-$w$-almost greedy basis is $(K+1)$-quasi-greedy.\end{lemma}

\begin{proof}Fix $m$ and $x$.  By Remark \ref{x-bound} together with $K$-$w$-almost-greediness,
$$\|x-G_m(x)\|\leq K\|x\|$$
and hence
$$\|G_m(x)\|\leq\|x-G_m(x)\|+\|x\|\leq(K+1)\|x\|.$$
\end{proof}

Let us give two facts from previous literature, which will be required momentarily.

\begin{proposition}[{\cite[Lemma 2.1]{DKKT}}]\label{constant-unconditional}
If $(e_n)_{n=1}^\infty$ is a $K$-quasi-greedy basis for a real Banach space then
$$\frac{1}{2K}\|\sum_{n\in A}e_n\|\leq\|\sum_{n\in A}\epsilon_ne_n\|\leq 2K\|\sum_{n\in A}e_n\|$$
for every choice of signs $(\epsilon_n)_{n=1}^\infty\in\{\pm 1\}^\mathbb{N}$ and every $A\in\mathbb{N}^{<\infty}$.  In this case we have
$$\|\sum_{n\in A}a_ne_n\|\leq 2K\|\sum_{n\in A}e_n\|\cdot\|(a_n)_{n\in A}\|_\infty$$
for every $(a_n)_{n=1}^\infty\in\mathbb{R}^\mathbb{N}$ and $A\in\mathbb{N}^{<\infty}$.\end{proposition}

\begin{proposition}[{\cite[Lemma 2.2]{DKKT}}]\label{Lemma-2.2}
If $(e_n)_{n=1}^\infty$ is a $K$-quasi-greedy basis for a real Banach space then
$$\|\sum_{n\in A}e_n\|\cdot\min_{n\in A}|a_n|\leq 4K^2\|\sum_{n\in A}a_ne_n\|$$
for every $(a_n)_{n=1}^\infty\in\mathbb{R}^\mathbb{N}$ and $A\in\mathbb{N}^{<\infty}$.\end{proposition}

The reader can also find the above Propositions in \cite{Tsparse}, p.65. 

\begin{theorem}\label{w-democratic}Every normalized $K$-$w$-almost greedy basis for a real Banach space is $w$-democratic with constant $\leq K$.\end{theorem}

\begin{proof}Let $A,B\subset\mathbb{N}^{<\infty}$ satisfy $w(A)\leq w(B)$.   Note that this also means $w(A\setminus B)\leq w(B\setminus A)$.  For $\delta>0$ we set
$$y:=\sum_{n\in A}e_n+(1+\delta)\sum_{n\in B\setminus A}e_n\;\;\;\text{ and }\;\;\;m:=|B\setminus A|).$$
By $K$-$w$-almost-greediness, we now have the estimates
\begin{align*}\|\sum_{n\in A}e_n\|
&=\|y-(1+\delta)\sum_{n\in B\setminus A}e_n\|
\\&=\|y-G_m(y)\|
\\&\leq K\|y-\sum_{n\in A\setminus B}e_n\|
\\&=K\|\sum_{n\in B}e_n+\delta\sum_{n\in B\setminus A}e_n\|
\\&\leq K\|\sum_{n\in B}e_n\|+\delta\|\sum_{n\in B\setminus A}e_n\|.
\end{align*}
%
As $\delta>0$ was arbitrary, we have
$$\|\sum_{n\in A}e_n\|\leq K\|\sum_{n\in B}e_n\|.$$
%
\end{proof}

\begin{theorem}\label{T2.5}If a basis $(e_n)_{n=1}^\infty$ for a real Banach space is both $K$-quasi-greedy and $D$-$w$-democratic, then it is $w$-almost greedy with constant $\leq 8K^4D+K+1$.\end{theorem}

\begin{proof}Fix $m$ and $x$, and suppose $A\in\mathbb{N}^{<\infty}$ satisfies $w(A)\leq w(\Lambda)$.  We claim that
$$\|x-G_m(x)\|\leq[4K^3(K+1)D+K+1]\cdot\|x-\sum_{n\in A}e_n^*(x)e_n\|,$$
so that, taking the infimum over all such $A$, the proof will be complete.

We begin with the estimate
\begin{align*}
\|x-G_m(x)\|
&=\|x-\sum_{n\in\Lambda}e_n^*(x)e_n\|
\\&\leq\|x-\sum_{n\in A}e_n^*(x)e_n\|+\|\sum_{n\in A}e_n^*(x)e_n-\sum_{n\in\Lambda}e_n^*(x)e_n\|
\\&=\|x-\sum_{n\in A}e_n^*(x)e_n\|+\|\sum_{n\in A\setminus\Lambda}e_n^*(x)e_n-\sum_{n\in\Lambda\setminus A}e_n^*(x)e_n\|
\\&\leq\|x-\sum_{n\in A}e_n^*(x)e_n\|+\|\sum_{n\in A\setminus\Lambda}e_n^*(x)e_n\|+\|\sum_{n\in\Lambda\setminus A}e_n^*(x)e_n\|.
\end{align*}

Observe that $w(A\setminus\Lambda)\leq w(\Lambda\setminus A)$.  Thus,
\begin{align*}\|\sum_{n\in A\setminus\Lambda}e_n^*(x)e_n\|
&\leq 2K\|\sum_{n\in A\setminus\Lambda}e_n\|\cdot\|(e_n^*(x))_{n\in A\setminus\Lambda}\|_\infty
&\text{(Proposition \ref{constant-unconditional})}
\\&\leq 2KD\|\sum_{n\in\Lambda\setminus A}e_n\|\cdot\|(e_n^*(x))_{n\in A\setminus\Lambda}\|_\infty
&\text{(}D\text{-}w\text{-democracy)}
\\&\leq 2KD\|\sum_{n\in\Lambda\setminus A}e_n\|\cdot\min_{n\in\Lambda\setminus A}|e_n^*(x)|
&\text{(definition of }\Lambda\text{)}
\\&\leq 8K^3D\|\sum_{n\in\Lambda\setminus A}e_n^*(x)e_n\|
&\text{(Proposition \ref{Lemma-2.2}).}
\end{align*}
Combining this with the last inequality gives
\begin{equation}\label{3}\|x-G_m(x)\|\leq\|x-\sum_{n\in A}e_n^*(x)e_n\|+(8K^3D+1)\|\sum_{n\in\Lambda\setminus A}e_n^*(x)e_n\|.\end{equation}

Finally, we have
$$\sum_{n\in\Lambda\setminus A}e_n^*(x)e_n=G_s\left(x-\sum_{n\in A}e_n^*(x)e_n\right),\;\;\;s:=|\Lambda\setminus A|,$$
and therefore, by $K$-quasi-greediness,
$$\|\sum_{n\in\Lambda\setminus A}e_n^*(x)e_n\|\leq K\|x-\sum_{n\in A}e_n^*(x)e_n\|.$$
We combine this with \eqref{3} to obtain
$$\|x-G_m(x)\|\leq(8K^4D+K+1)\|x-\sum_{n\in A}e_n^*(x)e_n\|.$$
\end{proof}

\begin{theorem} \label{mainequivalence} A basis for a real Banach space is $w$-almost greedy if and only if it is both quasi-greedy and $w$-democratic.\end{theorem}

\begin{remark}
Our proofs use only the following property of the $w$-measure:
 
$$ w(\emptyset) = 0\quad \text{and}\quad w(A) \le w(B) \Rightarrow w(A \setminus B) \le w(B \setminus A).$$
This property corresponds to a more general case than the one given by a $w$-measure. Indeed,
here is an example of a strictly monotone set function $\nu$ defined on all finite subsets of $\mathbb{N}$ satisfying \textbf{Property (*)}:
$$\nu(\emptyset) = 0\quad \text{and}\quad \nu(A) \le \nu(B) \Rightarrow \nu(A \setminus B) \le \nu(B \setminus A)$$ for which there does not exist a weight $w$ such that
\begin{equation} \label{eq: equiv} \nu(A) \le \nu(B) \Leftrightarrow w(A) \le w(B). \end{equation}
First we define $\nu$ on the power set  of $\{1,2,3\}$ by $\nu(\{1\}) =\nu(\{2\})=\nu(\{3\})=1/4$, $\nu(\{1,2\})=5/16$, $\nu(\{2,3\})=3/8$,
$\nu(\{1,3\})=7/16$, $\nu(\{1,2,3\}) = 1/2$. It is easily seen that $\nu$ is strictly monotone and  has Property (*) (restricted to subsets of $\{1,2,3\}$). 

Note that $\nu$ is not equivalent to
a weight $w$ because $\nu(\{1\}) =\nu(\{2\})=\nu(\{3\})$ would imply by \eqref{eq: equiv} that $w_1=w_2=w_3$ which would imply
that $w(\{1,2\})=w(\{1,3\})=w(\{2,3\})$ which in turn would imply by \eqref{eq: equiv} that  $\nu(\{1,2\})=\nu(\{1,3\})=\nu(\{2,3\}).$

Now we extend $\nu$ to all finite subsets $A$ of $\mathbb{N}$. Write $A = A_1 \cup A_2$ where $A_1= A \cap \{1,2,3\}$, $A_2=A\setminus A_1$. Define
$\nu(A) = \nu(A_1) + |A_2|$. Clearly $\nu$ is strictly monotone. 

Let us check Property (*). Suppose $\nu(A) \le \nu(B)$, where
$A = A_1 \cup A_2$ and $B = B_1 \cup B_2$. Then $\nu(A_1) + |A_2| \le \nu(B_1) + |B_2|$. Now $|\nu(A_1) - \nu(B_1| \le 1/2$
and hence $|A_2| \le |B_2|$. Also $\nu(A \setminus B)  = \nu(A_1 \setminus B_1) + |A_2 \setminus B_2|$ and
$\nu(B \setminus A)  = \nu(B_1 \setminus A_1) + |B_2 \setminus A_2|$.
 
We consider two cases. First, if 
$|A_2| = |B_2|$, then $\nu(A_1) \le \nu(B_1)$ and hence   $\nu(A_1 \setminus B_1) \le \nu(B_1 \setminus A_1)$ since $A_1,B_1 \subseteq \{1,2,3\}$. But also $|A_2 \setminus B_2| = |B_2 \setminus A_2|$ and hence $\nu(A \setminus B) \le \nu(B \setminus A)$.

Secondly, if $|A_2|  \le |B_2| - 1$, then $|A_2 \setminus B_2| \le |B_2 \setminus A_2|-1$ and hence
$$ \nu(A_1 \setminus B_1) + |A_2 \setminus B_2| \le 1/2 + |B_2 \setminus A_2|-1 < \nu(B_1 \setminus A_1) + |B_2 \setminus A_2|,$$
and hence $\nu(A \setminus B) \le \nu(B \setminus A)$.
\end{remark}

\section{Lebesgue-type inequalities} \label{LI} For $x\in X$ and $m\in \mathbb N$ denote
$$
\tilde \sigma_m(x) := \inf\{\|x-\sum_{n\in A} e_n^*(x)e_n\|\,:\, |A|=m\}
$$
the expansional best $m$-term approximation of $x$ with respect to $E:=\{e_n\}_{n=1}^\infty$. The following Lebesgue-type inequality was obtained in \cite{DKKT}: suppose $E$ is $K$-quasi-greedy and $D$-democratic, then
$$
\|x-G_m(x)\| \le C(K)D \tilde \sigma_m(x).
$$
This inequality was slightly generalized in \cite{Tsparse}, p.91, Theorem 3.5.2. Suppose that for a basis $E$ there exists an increasing sequence $v(m):=v(m,E)$ such that, for any two sets of indices $A$ and $B$ with $|A|\le|B|\le m$, we have
\begin{equation}\label{2.3}
\|\sum_{n\in A} e_n\| \le v(m) \|\sum_{n\in B} e_n\|.
\end{equation}
The following statement is from \cite{Tsparse}, p.91, Theorem 3.5.2.

\begin{proposition}\label{P2.4} Let $E$ be a $K$-quasi-greedy basis of $X$ satisfying (\ref{2.3}). Then, for each $x\in X$,
$$
\|x-G_m(x)\| \le C(K)v(m) \tilde \sigma_m(x).
$$
\end{proposition}

The above proof of Theorem \ref{T2.5} gives us the following version of Proposition \ref{P2.4} 
in the weighted case. Suppose that for given weight sequence $w=\{w_n\}_{n=1}^\infty$ and basis $E$ there exists an increasing function $v(u,w):=v(u,E,w)$ such that, for any two sets of indices $A$ and $B$ with $w(A)\le w(B)\le u$, we have
\begin{equation}\label{2.4}
\|\sum_{n\in A} e_n\| \le v(u,w) \|\sum_{n\in B} e_n\|.
\end{equation}

For $x\in X$ and $u>0$ denote for a given weight sequence $w$
$$
\tilde \sigma^w_u(x) := \inf\{\|x-\sum_{n\in A} e_n^*(x)e_n\|\,:\, w(A)\le u\}.
$$

\begin{theorem}\label{T2.6} Let $E$, a $K$-quasi-greedy basis of $X$, and a weight sequence $w$ satisfy (\ref{2.4}). Then, for each $x\in X$,
\begin{equation}\label{2.5}
\|x-G_m(x)\| \le C(K)v(w(\Lambda_m),w) \tilde \sigma^w_{w(\Lambda_m)}(x),
\end{equation}
where $\Lambda_m$ is from
$$
G_m(x) = \sum_{n\in \Lambda_m} e_n^*(x)e_n.
$$
\end{theorem}

We note that the left hand side of (\ref{2.5}) does not depend on the weight sequence $w$, but the right hand side of (\ref{2.5}) does depend on $w$. Therefore, our setting with weights allows us to estimate $\|x-G_m(x)\|$ choosing different weight sequences $w$ and optimize over them. 

Theorem \ref{T2.6} gives a Lebesgue-type inequality in terms of expansional best $m$-term 
approximation. We now consider a question of replacing the expansional best $m$-term 
approximation by the best $m$-term approximation
$$
 \sigma^w_u(x) := \inf\{\|x-\sum_{n\in A} c_ne_n\|\,:\, c_n, n\in A, \, w(A)\le u\}
$$
where inf is taken over all sets of indices $A$ with $w$-measure $w(A)\le u$ and all coefficients $c_n$, $n\in A$. In case $w_i=1$, $i\in \mathbb N$, we drop $w$ from the notation. 
It is clear that
$$
\sigma^w_u(x,E) \le \tilde\sigma^w_u(x,E).
$$
It is also clear that for an unconditional basis $E$ we have
$$
\tilde\sigma^w_u(x,E) \le C(X,E)\sigma^w_m(x,E).
$$

We recall some known useful properties of quasi-greedy bases.  For a given element $x \in X$ we consider
the expansion
$$
x=\sum_{k=1}^{\infty}e^*_k(x)e_k.
$$
Let a sequence $k_j,j=1,2,...,$ of 
positive integers be such that
$$
|e^*_{k_1}(x)|\geq |e^*_{k_2}(x)|\geq... \,\,.
$$
We use the notation
$$
a_j(x):=|e^*_{k_j}(x)|
$$
for the decreasing rearrangement of the coefficients of $x$. 
It will be convenient in this section to redefine the quasi-greedy constant $K$ to be the least constant such that
$$
\|G_m(f)\|\le K\|f\|\quad\text{and}\quad \|f-G_m(f)\|\le K\|f\|, \quad f\in X.
$$

The following Lemma \ref{L2.2n} is from \cite{DKKT} (see also \cite{Tsparse}, p.66).

\begin{lemma}\label{L2.2n} Suppose $E = \{e_n\}_{n\in\mathbb N}$ has
quasi-greedy constant $K$. For any $x\in X$ and any $n\in\mathbb N$ we have
\begin{equation}\label{2.7n}
a_n(x)\|\sum_{j=1}^ne_{k_j}\| \le 4K^2\|x\|.
\end{equation}
\end{lemma}

For a set of indices $\Lambda$ define
$$
S_\Lambda(x,E) := \sum_{k\in \Lambda} e^*_k(x)e_k.
$$
The following Lemma \ref{L2.3t} is from \cite{DKK} (see also \cite{Tsparse}, p.67).

\begin{lemma}\label{L2.3t} Let $E$ be a quasi-greedy basis of $X$. Then for any finite set of indices $\Lambda$, $|\Lambda|=m$, we have for all $x\in X$
$$
\|S_\Lambda(x,E)\| \le C \ln(m+1)\|x\|.  
$$
\end{lemma}

Lemma \ref{L2.3t} was used in \cite{GHO} and \cite{DS-BT} to prove the following inequality.

\begin{lemma}\label{L2.4t} Let $E$ be a quasi-greedy basis of $X$. Then for all $x\in X$ 
$$
\tilde\sigma_m(x) \le C\ln(m+1) \sigma_m(x).
$$
\end{lemma}

We prove an estimate for $\tilde\sigma^w_n(x,E)$ in terms of $\sigma^w_m(x,E)$ for a quasi-greedy basis $E$. For a basis $E$ we define the fundamental function
$$
\varphi^w (u) := \sup _{A:w(A)\le u} \|\sum_{k\in A} e_k\|.
$$
We also need the following function
$$
  \phi^w(u):= \inf _{A:w(A) > u} \|\sum_{k\in A} e_k\|.
$$
In the case $w_i=1$, $i\in \mathbb N$, we drop $w$ from the notation.
The following inequality was obtained in \cite{DKKT}.
\begin{lemma}\label{L2.4kt} Let $E$ be a quasi-greedy basis. Then for any $m$ and $r$ there exists a set $F$, $|F|\le m+r$ such that
$$
\|x-S_F(x)\| \le C\left(1+\frac{\varphi(m)}{\phi(r)}\right)\sigma_m(x)
$$
and, therefore, for any $x\in X$
$$
\tilde \sigma_{m+r}(x) \le C\left(1+\frac{\varphi(m)}{\phi(r)}\right)\sigma_m(x).
$$
\end{lemma}

We now prove the following weighted version of Lemma \ref{L2.4kt}.

\begin{lemma}\label{L2.4w} Let $E$ be a quasi-greedy basis. Then for any $u$ and $v$  we have for each $x\in X$
$$
\tilde \sigma^w_{u+v}(x) \le C\left(1+\frac{\varphi(u)}{\phi(v)}\right)\sigma^w_u(x).
$$
\end{lemma}
\begin{proof} For an arbitrary $\epsilon>0$,   let  $A$ be a set, $w(A)\le u$, and $p_u(x)$ be a polynomial such that
\begin{equation}\label{1.8kt}
\|x-p_u(x)\|\le \sigma^w_u(x)+\epsilon, \quad p_u(x) = \sum_{k\in A}b_ke_k. 
\end{equation}
Denote $g:= x-p_u(x)$. Let $m$ be such that for the set $B:=\Lambda_m(g)$ we have  $w(B)\le v$ and $w(\Lambda_{m+1}(g))>v$. Consider
$$
G_m(g) = \sum_{k\in B}e^*_k(g)e_k.
$$
We have
\begin{equation}\label{1.9kt}
x-S_{A\cup B}(x) = g-S_{A\cup B}(g) = g-S_B(g) -S_{A\setminus B}(g).
 \end{equation}
By the assumption that $E$ is quasi-greedy and by the definition of $B$ we get
\begin{equation}\label{1.10kt}
\|g-S_B(g)\| \le C_1\|g\| \le C_1(\sigma^w_u(x)+\epsilon).  
\end{equation}
Let us estimate $\|S_{A\setminus B}(g)\|$.  By Lemma \ref{L2.2n} we get
$$
\max_{k\in A\setminus B}|e^*_k(g)|\le a_{m+1}(g) \le 4K^2(\phi(v))^{-1}\|g\|.
$$
Next, by Proposition \ref{constant-unconditional} we obtain
\begin{equation}\label{1.11kt}
\|S_{A\setminus B}(g)\| \le (2K)^3 \varphi(u)\phi(v)^{-1}\|g\|.  
\end{equation}
Combining (\ref{1.10kt}) and (\ref{1.11kt}) we derive from (\ref{1.9kt}) for $F:=A\cup B$
$$
\|f-S_F(x)\| \le C(K)\left(1+\frac{\varphi(u)}{\phi(v)}\right)(\sigma^w_u(x)+\epsilon).
$$
Lemma \ref{L2.4w} is proved.
\end{proof}

\begin{theorem}\label{T3.1} Let $E$ be a $K$-$w$-quasi-greedy and $D$-$w$-democratic. Then, for any $x\in X$, we have
$$
\|x-G_m(x)\| \le C(K,D)\sigma^w_{w(\Lambda_m)/2}(x).
$$
\end{theorem} 
\begin{proof} 
By Theorem \ref{T2.5} we have
\begin{equation}\label{3.9}
\|x-G_m(x)\| \le C(K)D\tilde \sigma^w_{w(\Lambda_m)}(x).
\end{equation}
Using inequality $\varphi^w(u)\le D\phi^w(u)$, by Lemma \ref{L2.4w} with $u=v=w(\Lambda_m)/2$ we obtain
\begin{equation}\label{3.10}
 \tilde \sigma^w_{w(\Lambda_m)}(x) \le C(K)(1+D)\sigma^w_{w(\Lambda_m)/2}(x).
\end{equation}
Combining (\ref{3.9}) and (\ref{3.10}) we complete the proof of Theorem \ref{T3.1}.

\end{proof}

\section{$w$-Semi-greedy bases} \label{sec: semigreedy}
In this section
we consider an obvious enhancement of the TGA which improves
the rate of convergence. Suppose that $x \in X$ and let $\rho$
be the greedy ordering for $x$.
Let $\overline{G}_m(x) \in \operatorname{span}\{e_{\rho(n)} \colon 1 \le n \le m\}$
be a Chebyshev approximation to $x$. Thus,
\begin{equation*}
\|x-\overline{G}_m(x)\| = \min\{\|x - \sum_{n=1}^m a_n e_{\rho(n)}\| \colon (a_n)_{n=1}^m \in \mathbb{R}^m\}. \end{equation*}
For $y>0$, let
$$\sigma^w_y(x) := \inf \{ \|x - \sum_{n \in A} a_n e_n\| \colon w(A) \le y, a_n\in \mathbb{R}\}$$ 
denote the error in the best approximation to $x$ by vectors of support of weight at most $y$.
\begin{definition}We say that a basis $(e_n)_{n=1}^\infty$ for a Banach space $X$ is {\bf $\boldsymbol{w}$-semi-greedy with constant $\boldsymbol{\overline{K}}$} if
$$\|x-\overline{G}_m(x)\|\leq \overline{K}\sigma^w_{w(\Lambda_m)}(x)$$
for all $m\in\mathbb{N}$ and $x\in X$.
\end{definition}

Our first goal is to prove that every $w$-almost greedy basis is $w$-semi-greedy, which generalizes \cite[Theorem 3.2]{DKK}.
To that end, we recall  an important property of the `truncation function' proved in \cite{DKK}.

Fix $M>0$. Define the truncation function $f_M\colon \mathbb{R} \rightarrow [-M,M]$
thus: \begin{equation*}
f_M(x) = \begin{cases} M &\text{for $x>M$};\\
x &\text{for $x \in [-M,M]$};\\
-M &\text{for $x<-M$}.\end{cases}
\end{equation*}

 \begin{proposition}[{\cite[Proposition 3.1]{DKK}}] \label{prop: cutoff}
 Suppose that $(e_n)$ is $K$-quasi-greedy. Then, for every $M>0$ and for all real scalars $(a_n)$, we have
\begin{equation*}
\|\sum_{i=1}^\infty f_M(a_n) e_n\| \le
(1+3K)\|\sum_{n=1}^\infty a_n e_n\|. \end{equation*}
\end{proposition}
\begin{theorem} \label{thm: almostgreedysemigreedy}
Every $w$-almost greedy basis of a real Banach space $X$ is $w$-semi-greedy. \end{theorem}
\begin{proof} By Lemma~\ref{quasi-greedy} and Theorem~\ref{w-democratic}, $(e_n)$ is quasi-greedy
and $w$-democratic. Let $K$ and $D$ the quasi-greedy and
democratic constants of $(e_n)$, respectively.
 Fix $m\ge1$ and $x = \sum_{n=1}^\infty a_n e_n$ in $X$. Let $\Lambda := \Lambda_{m,\tau,x}$.
Let $z := \sum_{n \in B} b_n e_n$, where $w(B) \le w(\Lambda)$, satisfy
\begin{equation*}\|x-z\| \le 2\sigma^w_{w(\Lambda)}(x). \end{equation*}
If $B=\Lambda$ then there is nothing to prove. So we may assume (since $w(B) \le w(\Lambda)$)
that $\Lambda \setminus B$ is nonempty. Let $k = |\Lambda \setminus B|$, so that $1\le k\le m$,
and let $M = |a_{\rho(m)}|$. Then, using both parts of Proposition~\ref{constant-unconditional},
\begin{equation} \label{eq: x-z1}
M\| \sum_{n \in \Lambda\setminus B} e_n\|\le 2KM\|\sum_{n \in \Lambda} e_n\| \le 4K^2 \|x-z\|, \end{equation}
since $|e^*_n(x-z)| \ge M$ for all $n \in \Lambda \setminus B$.  Let
\begin{equation*}
x-z:=\sum_{n=1}^\infty y_ne_n. \end{equation*}
 By Proposition~\ref{prop: cutoff}, we have
 \begin{equation} \label{eq: x-z2}
\|\sum_{n=1}^\infty f_M(y_n)e_n\| \le (1+3K)\|x-z\|. \end{equation}
Note that
\begin{align*}
v:&= \sum_{n \in \Lambda}f_M(y_n)e_n + \sum_{n \in \mathbb{N}\setminus \Lambda}a_ne_n\\
&= \sum_{n=1}^\infty f_M(y_n) e_n +
\sum_{n \in B\setminus \Lambda}(a_n - f_M(y_n))e_n\end{align*}
since $a_n = y_n = f_M(y_n)$ for all $n \in \mathbb{N} \setminus (\Lambda \cup B)$
Hence, \begin{align*}
\|v\| &\le \|\sum_{n=1}^\infty f_M(y_n) e_n\| +
 \|\sum_{n \in B\setminus \Lambda}(a_n - f_M(y_n))e_n\|\\
&\le (1+3K)\|x-z\| + 4KM \|\sum_{n \in B \setminus \Lambda} e_n\|
\intertext{(by \eqref{eq: x-z2} and by
Proposition~\ref{constant-unconditional}, since $|a_n - f_M(y_n)| \le 2M$ for all $n\in B\setminus \Lambda$)}
&\le (1+3K)\|x-z\| + 4KDM \|\sum_{n \in \Lambda \setminus B} e_n\|\\
\intertext{(since $w(B\setminus \Lambda)\le w(\Lambda \setminus B)$ and $(e_n)$ is $w$-democratic)}
&\le (1 + 3K + 16K^3D)\|x-z\|\\
 \end{align*} by \eqref{eq: x-z1}.  Taking the infimum  over all $z$ gives 
$$ \|v\| \le (1 + 3K + 16K^3D) \sigma^w_{w(\Lambda)}(x)$$
Since $v=x-\sum_{n\in \Lambda}
(a_n-f_M(y_n))e_n$,   we conclude that $(e_n)$ is semi-greedy with constant $1 + 3K + 16K^3D$.
\end{proof}
\begin{corollary} Suppose $(e_n)$ is $w$-almost greedy. Then, for all $x \in X$ and $m \ge 1$, 
$$\| x - G_m(x)\| \le (1 + 3K + 16K^3D) \sigma^w_{w(\Lambda)}(x) + 2K |e^*_{\rho(m)}(x)| \|\sum_{n \in \Lambda} e_n\|. $$
\end{corollary} 
\begin{proof} Using the notation of the last result,  note that  $x - G_m(x) = v - \sum_{n \in \Lambda} f_M(y_n) e_n$. Hence
\begin{align*} \|x - G_m(x)\| &\le \|v\| + \|\sum_{n \in \Lambda} f_M(y_n) e_n\|\\
&\le (1 + 3K + 16K^3D)\|x-z\| + 2K|e^*_{\rho(m)}(x)| \|\sum_{n \in \Lambda} e_n\|, \end{align*}
by Proposition~\ref{constant-unconditional} recalling that $M = |e^*_{\rho(m)}(x)|$. Now take the infimum of $\|x-z\|$ over $z$
to get the result.
\end{proof}

The remainder of this section investigates the converse of Theorem~\ref{thm: almostgreedysemigreedy}.
To that end we first show that certain properties of the weight sequence $(w_n)$ imply the existence of  subsequences of $(e_n)$ that are equivalent to the unit vector basis of $c_0$. 

\begin{proposition} \label{prop: weightproperties}  Let $(e_n)$ be a $w$-semi-greedy basis with constant $\overline{K}$ and let $\beta$ be the basis constant.  Suppose $A \in \mathbb{N}^{<\infty}$. \begin{itemize} 
\item[(i)] If $w(A) \le \limsup_{n \rightarrow \infty} w_n$ then $\max_\pm\|\sum_{n \in A} \pm e_n\| \le 2\beta \overline{K}$.
\item[(ii)] If $\sum_{n=1}^\infty w_n <\infty$ then $(e_n)$ is equivalent to the unit vector basis of $c_0$.
\item[(iii)] If $\sup w_n = \infty$ then $(e_n)$ is equivalent to the unit vector basis of $c_0$.
\item[(iv)] If $\inf w_n =0$ then $(e_n)$ contains a subsequence equivalent to the unit vector basis of $c_0$.
\end{itemize}\end{proposition}
\begin{proof} (i) We may select $n_1>n_0 > \max A$ such that $w(A) < \delta := w_{n_0} + w_{n_1}$.  Let $\varepsilon>0$. The $w$-semi-greedy condition applied to 
$ x := \sum_{n \in A} \pm e_n + (1+ \varepsilon)(e_{n_0}+ e_{n_1})$ implies the existence of $\lambda, \mu\in \mathbb{R}$ such that
$$\|\sum_{n \in A} \pm e_n + \lambda e_{n_0}+\mu e_{n_1}\| \le \overline{K}\sigma^w_{\delta}(x).$$
Hence \begin{align*}
\|\sum_{n \in A} \pm e_n \| &\le \beta \|\sum_{n \in A} \pm e_n + \lambda e_{n_0}+ \mu e_{n_1}\|\\
&\le \beta \overline{K} \sigma^w_{\delta}(x)\\ &\le \beta \overline{K} (1+ \varepsilon) \| e_{n_0}+ e_{n_1}\|\\
\intertext{(since $w(A) \le \delta$)}& \le 2\beta \overline{K} (1+ \varepsilon). \end{align*}
Let $\varepsilon \rightarrow 0$ to conclude.

(ii) Choose  $N \in \mathbb{N}$ such that $\sum_{N+1}^\infty w_n < w_1$. Suppose $\min A \ge N+1$. The $w$-semi-greedy condition applied to 
$ x := (1+ \varepsilon) e_1 + \sum_{n \in A} \pm e_n$ implies the existence of $\lambda\in \mathbb{R}$ such that
$$\|\lambda e_1 + \sum_{n \in A} \pm e_n \| \le \overline{K}\sigma^w_{w_1}(x).$$ Hence \begin{align*}
\|\sum_{n \in A} \pm e_n \| &\le 2\beta \|\lambda e_1 + \sum_{n \in A} \pm e_n\|\\
&\le 2\beta \overline{K} \sigma^w_{w_{1}}(x)\\ &\le 2\beta \overline{K} (1+ \varepsilon) \| e_{1}\|\\
\intertext{(since $w(A) \le w_1$)}& = 2\beta \overline{K} (1+ \varepsilon). \end{align*}
Hence $(e_n)$ is equivalent to the unit vector basis of $c_0$.

(iii) By (i), $\| \sum_{n \in A} \pm e_n \| \le 2\beta \overline{K}$ for all $A$ and choices of signs. Hence $(e_n)$ is equivalent to the unit vector basis of $c_0$.

(iv) Choose $(n_k)$ such that $\sum_{k=1}^\infty w_{n_k} < \infty$. By (ii), $(e_{n_k})$ is equivalent to the unit vector basis of $c_0$.
\end{proof}

The following definition generalizes the notion of superdemocracy to weights, with the constant weight corresponding to the usual definition of superdemocracy (see \cite{KT}).

\begin{definition}A basis $(e_n)_{n=1}^\infty$ is {\bf $\boldsymbol{w}$-superdemocratic with constant $\boldsymbol{\overline{D}}$} whenever $w(A)\leq w(B)$ for $A,B\in\mathbb{N}^{<\infty}$ implies
$$\max_\pm\|\sum_{n\in A} \pm e_n\|\leq \overline{D} \min_\pm\|\sum_{n\in B} \pm e_n\|.$$
\end{definition}

Combining the fact that quasi-greedy sequences are `unconditional for constant coefficients' (Proposition~\ref{constant-unconditional}) with Theorem~\ref{mainequivalence} immediately  yields the following result.

\begin{proposition}  \label{prop: almostgreedyissuperdemocratic}  Every $w$-almost greedy basis of a real Banach space is $w$-superdemocratic. \end{proposition} 
 
We show next that $w$-semi-greedy bases are also $w$-superdemocratic, which is the weighted version of \cite[Proposition 3.3]{DKK}.

\begin{theorem} Every $w$-semi-greedy basis of a real Banach space is $w$-superdemocratic. \end{theorem}
\begin{proof} We may assume that $\sum w_n = \infty$ and $\sup w_n < \infty$, for otherwise by Proposition~\ref{prop: weightproperties} $(e_n)$ is equivalent to the unit vector basis of 
$c_0$ for which the result is obvious. Suppose $w(A) \le w(B)$ and that $B$ is nonempty.  If $w(B) \le \limsup w_n$ then by Proposition~\ref{prop: weightproperties} ,
$$\max_\pm \| \sum_{n \in A} \pm e_n \| \le 2\beta \overline{K},$$ and hence
$$\max_\pm \| \sum_{n \in A} \pm e_n \| \le 2\beta \overline{K}\min_\pm \| \sum_{n \in B} \pm e_n \|$$
as desired. Now suppose that $w(B) > \limsup w_n$. Since $\sum w_n = \infty$ we can choose $E \in \mathbb{N}^{<\infty}$ and $n_0 \in \mathbb{N}$ with $\min E > \max(A \cup B)$ and $n_0 > \max E$ such that $$w(E) \le w(B) < w(E) + w_{n_0} < 2 w(B).$$
Set $F := E \cup \{n_0\}$. Then $w(E) \le w(B) < w(F) < 2w(B)$.  Let $\varepsilon>0$. Applying the $w$-semi-greedy condition to
$$ x = (1+\varepsilon) \sum_{n \in B} \pm e_n + \sum_{n \in E} e_n $$
yields scalars $(a_n)_{n \in B}$ such that \begin{align*}
\| \sum_{n \in B} a_n e_n + \sum_{n \in E} e_n\| &\le \overline{K} \sigma^w_{w(B)}(x)\\
&\le \overline{K}(1+\varepsilon) \|\sum_{n \in B} \pm e_n\|\end{align*}
since $w(E) \le w(B)$. Hence $$
\|\sum_{n \in E} e_n \| \le 2\beta \| \sum_{n \in B} a_n e_n + \sum_{n \in E} e_n\| \le 2\beta \overline{K}(1+\varepsilon) \|\sum_{n \in B} \pm e_n\|.$$
Since $\varepsilon>0$ is arbitrary, $$\|\sum_{n \in E} e_n \| \le 2\beta \overline{K}\|\sum_{n \in B} \pm e_n\|.$$
Similarly, using the fact that $w(A)
 \le w(B) < w(F)$, the $w$-semi-greedy condition applied to $y =  \sum_{n \in A} \pm e_n + (1 + \varepsilon) \sum_{n \in F} e_n $
yields $$\|\sum_{n \in A} \pm e_n \| \le \beta \overline{K}\|\sum_{n \in F} e_n\|.$$ Finally,
\begin{align*}
\|\sum_{n \in A} \pm e_n \| &\le \beta \overline{K}\|\sum_{n \in F} e_n\|\\
&\le \beta \overline{K}(\|\sum_{n \in E} e_n\| +1)\\
&\le \beta \overline{K}(2\beta \overline{K}\|\sum_{n \in B} \pm e_n\| +1)\\
&\le (2\beta^2\overline{K}^2 + \beta \overline{K})\|\sum_{n \in B} \pm e_n\|. \end{align*}

\end{proof}
\begin{proposition} \label{prop: superdemocraticequiv} Suppose $0 < \inf w_n \le \sup w_n < \infty$. Then $(e_n)$ is $w$-superdemocratic $\Leftrightarrow$ $(e_n)$ is superdemocratic for the constant weight sequence. \end{proposition}
\begin{proof} $\Rightarrow$: Let $\overline{D}$ be the $w$-superdemocracy constant of $(e_n)$. We may assume that $0 < \alpha := \inf w_n \le 1 = \sup w_n$. Suppose $|A| = |B|$ and, without loss of generality,  $w(A) \le w(B)$. Then
$$\max_\pm \| \sum_{n \in A} \pm e_n \| \le \overline{D} \min_\pm \| \sum_{n \in B} \pm e_n \|$$
So to prove superdemocracy it  suffices to show that 
$$\max_\pm \| \sum_{n \in B} \pm e_n \| \le L \min_\pm \| \sum_{n \in A} \pm e_n \|$$
for some constant $L$.
If $w(B) \le 2/\alpha$ then $|B| \le 2/\alpha^2$ and so we can take  $L = 2/\alpha^2$. Suppose $w(B) > 2/\alpha$. Note that $w(A) \ge \alpha w(B)\ge 2$.
Hence we may partition $B$ into $N$ sets $B_1,\dots, B_N$ satisfying $w(B_j) \le w(A) \le w(B_j)+1$, and hence $w(B_j) \ge w(A)/2$,
with $$ N \le \frac{w(B)}{w(A)/2} \le \frac{2}{\alpha}.$$ Since $(e_n)$ is $w$-superdemocratic and $w(B_j) \le w(A)$ ($1 \le j \le N$),
we have
\begin{align*} \max_{\pm}\| \sum_{n \in B} \pm e_n \| &\le \sum_{j=1}^N \max_{\pm}\| \sum_{n \in B_j} \pm e_n \|\\
&\le N\overline{D}  \min_{\pm}\| \sum_{n \in A} \pm e_n \|\\
&\le \frac{2\overline{D}}{\alpha} \min_{\pm}\| \sum_{n \in A} \pm e_n \|
\end{align*} Hence we can take $L = 2\overline{D}/\alpha$.

$\Leftarrow$: Let $C$ be the superdemocracy constant (for the constant weight). Suppose $w(A) \le w(B)$. Then $|A| \le |B|/\alpha$.
We can partition $A$ into fewer than $1 + 1/\alpha$ sets of size at most $|B|$. Hence by the triangle inequality
$$\max_{\pm}\| \sum_{n \in A} \pm e_n \| \le  \frac{C(1+\alpha)}{\alpha}\min_{\pm}\| \sum_{n \in B} \pm e_n \|.$$
So $(e_n)$ is $w$-superdemocratic.
\end{proof} \begin{remark} The previous result is sharp in the following sense. Suppose $w$ is a weight sequence satisfying
$\inf w_n = 0$, $\sup w_n =1$, and $\sum w_n = \infty$. Consider the following norm:
$$ \| \sum_{n=1}^\infty a_n e_n \| = \sup |a_n| \vee (\sum_{n=1}^\infty a_n^2 w_n)^{1/2}.$$
Then $(e_n)$ is a normalized basis which is $w$-superdemocratic but not superdemocratic (and hence $(e_n)$ is $w$-greedy but not greedy). \end{remark}
\begin{corollary} \label{cor: almostgreedyequivwalmostgreedy} Suppose that $(e_n)$ has no subsequence equivalent to the unit vector basis of $c_0$. Then $(e_n)$ is $w$-almost greedy
for some weight $w$ $\Leftrightarrow$ $(e_n)$ is almost greedy. \end{corollary}
\begin{proof} 
By Theorem~\ref{mainequivalence} and Proposition~\ref{prop: almostgreedyissuperdemocratic}, $(e_n)$ is $w$-almost greedy $\Leftrightarrow$ $(e_n)$ is quasi-greedy and $w$-superdemocratic. Since $(e_n)$ has no subsequence equivalent to the unit vector basis of $c_0$, by Proposition~\ref{prop: weightproperties} (parts $(iii)$ and $(iv)$), $0 < \inf w_n \le \limsup w_n < \infty$. Hence, by Proposition~\ref{prop: superdemocraticequiv},
 $(e_n)$ is $w$-almost greedy $\Leftrightarrow$ $(e_n)$ is quasi-greedy and superdemocratic $\Leftrightarrow$ $(e_n)$ is almost greedy.
\end{proof} \begin{question} \label{question: semigreedy} Does the analogous result holds for semi-greediness?
A simple characterization of $w$-semi-greediness analogous to Theorem~\ref{mainequivalence}, which could be useful for answering this question, seems to be lacking. A partial answer is given in Corollary~\ref{cor: semigreedywsemigreedyequiv} below.\end{question}
We turn now to a  converse to Theorem~\ref{thm: almostgreedysemigreedy} for spaces of finite cotype. The following lemma is needed.
\begin{lemma} \label{lem: lowerest} Suppose $(e_n)$ is $w$-semi-greedy and $0 < \inf w_n  \le \sup w_n$. There exists  $M<\infty$ such that for all
$A \in \mathbb{N}^{<\infty}$ and for all scalars $(a_n)_{n \in A}$, we have
$$\min_{n \in A} |a_n| \| \sum_{n \in A} e_n \| \le M  \| \sum_{n \in A} a_n  e_n \|.$$\end{lemma} 
\begin{proof} We may assume $\sup w_n = 1$ and $\inf w_n = \alpha > 0$. 
By Proposition~\ref{prop: superdemocraticequiv}, $(e_n)$ is superdemocratic.
We may assume $w(A) > 2$, for otherwise $|A| \le 2/\alpha$ and
the result is clear. Choose $F$ with $\min F > \max A$ such that
$w(F) \le w(A) \le w(F) + 1$. Note that $$|F| \ge w(A) - 1 \ge \frac{w(A)}{2} \ge \frac{\alpha |A|}{2}.$$
Hence $$\| \sum_{n \in F} e_n \| \ge \frac{\alpha}{(2 +\alpha)D}\| \sum_{n \in A} e_n \|,$$ 
where $D$ is the democracy constant. Applying the semi-greedy condition to
$x = \sum_{n \in A}a_n e_n + (\min |a_n|) \sum_{n \in F} e_n$, there exist scalars $(c_n)_{n \in A}$ such that
\begin{align*}
\|\sum_{n \in A}c_n e_n + (\min |a_n|) \sum_{n \in F} e_n\| &\le \overline{K} \sigma^w_{w(A)}\\
&\le \overline{K} \| \sum_{n \in A}a_n e_n\|. \end{align*}
since $w(F) \le w(A)$. Let $\beta$ be the basis constant of $(e_n)$. Hence \begin{align*}
(\min |a_n|) \|\sum_{n \in A} e_n\| &\le (\frac{2}{\alpha} +1)D (\min |a_n|) \|\sum_{n \in F} e_n\| \\
&\le  2\beta(\frac{2}{\alpha} +1)D  \|\sum_{n \in A}c_n e_n + (\min |a_n|) \sum_{n \in F} e_n\|\\
&\le 2\beta(\frac{2}{\alpha} +1)D\overline{K}\| \sum_{n \in A}a_n e_n\|.
\end{align*}
\end{proof}

Let us recall that a Banach space $X$ has cotype $q$, where $2\le
q <\infty$, if there exists a constant $C$ such that
\begin{equation*} \label{eq: cotypeq}
(\sum_{j=1}^n\|x_j\|^q)^{\frac{1}{q}} \le C (\Ave_{\varepsilon_j=\pm1}
\|\sum_{j=1}^n\varepsilon_jx_j\|^q)^{\frac{1}{q}} \end{equation*} for
all $x_1,\ldots,x_n\in X$ and $n\in\mathbb{N}$.  The least such
 constant $C$ is called the
cotype $q$-constant $C_q(X)$. We say that $X$ has \textit{finite cotype} if $X$ has cotype $q$ for some $q<\infty$.

The following result  and its proof can be extracted from \cite[p. 76]{DKK}. It is used below in the proof of Theorem~\ref{thm: finitecotypeequiv}.
\begin{proposition} \label{prop:largekquasigreedy} Suppose that $X$ has finite cotype and that $(e_n)$ is superdemocratic. Then, for all $0<\theta <1$, there 
exists $L(\theta)< \infty$ such that for all $m \ge 1$ and for all $x = \sum_{n \in F} a_n e_n$, where $|F| = m$, and for all $n \ge \theta m$, we have $\|G_n(x)\| \le L(\theta) \|x\|$.  ($L(\theta)$ also depends on $X$ and on $(e_n)$ but we suppress this dependence.) \end{proposition}

Now we can prove the weighted version of \cite[Theorem 3.6]{DKK}, which provides a partial converse to Theorem~\ref{thm: almostgreedysemigreedy}.
\begin{theorem} \label{thm: finitecotypeequiv} Suppose that $X$ has finite cotype and that $(e_n)$ is $w$-semi-greedy. Then $(e_n)$ is $w$-almost greedy.
\end{theorem} \begin{proof} Since $X$ has finite cotype, no subsequence of $(e_n)$ is equivalent to the unit vector basis of $c_0$. Hence we may assume that $\sup w_n = 1$ and $\inf w_n = \alpha>0$.
By Proposition~\ref{prop: superdemocraticequiv}, $(e_n)$ is both $w$-superdemocatic and superdemocratic. Hence,
 by Theorem~\ref{mainequivalence}, it suffices to show that $(e_n)$ is quasi-greedy, i.e., that there exists $K < \infty$ such that  for all $k \ge 1$ and $x \in X$,  $\|G_k(x)\| \le K \|x\|$.

Fix $m \ge 1$. Suppose that $x = \sum_{n \in F} a_n e_n$ with $\|x\| =1$, where $|F| = m$ and $a_n \ne 0$ for $n \in F$. Fix $k \ge 1$ and let $A := \{\rho(1),\dots, \rho(k)\}$. Choose $E \in \mathbb{N}^{<\infty}$
such that $\min E > \max F$ and $w(A) \le w(E) \le w(A) +1$. This implies that
\begin{equation} \label{eq: cardE} \alpha k = \alpha |A| \le |E|\le \frac{|A| + 1}{\alpha} = \frac{k+1}{\alpha} \le \frac{2k}{\alpha}.
\end{equation}

 Now let $B:= \{\rho(k+1),\dots, \rho(k+ l))\}$, where $w(A) \le w(B) \le w(A) + 1$, so that 
 $$ \alpha k = \alpha |A| \le |B|= l\le \frac{|A| + 1}{\alpha} = \frac{k+1}{\alpha} \le \frac{2k}{\alpha}.$$

If $k+l > m$, then $k >  \alpha m /(2 + \alpha)$, and hence by Proposition~\ref{prop:largekquasigreedy}
$$\|G_k(x)\| \le L(\alpha/(2+\alpha)) \|x\|.$$
So we may assume $k + l \le m$.

Fix $\varepsilon>0$ and consider
$$ y := \sum_{n \in F\setminus A} a_ne_n + (|a_{\rho(k)}|+\varepsilon)(\sum_{n \in E} e_n).$$
Since $w(A) \le w(E)$, we have
\begin{equation} \label{eq: sigmaky}
 \sigma^w_{w(E)}(y) \le \|x+(|a_{\rho(k)}|+\varepsilon)(\sum_{n \in E} e_n)\| \le 1
+(|a_{\rho(k)}|+\varepsilon)\|\sum_{n \in E} e_n\|. \end{equation}
Since $(e_i)$ is semi-greedy there exist scalars $(c_n)$ ($n \in E$) such that
\begin{equation} \label{eq: sigmaky2}
\|\sum_{n \in F\setminus A} a_ne_n + \sum_{n \in E} c_ne_n\| \le \overline{K}\sigma^w_{w(E)}(y).
\end{equation}
Since $\max F<\min E$ and $\varepsilon>0$ is arbitrary,
\eqref{eq: sigmaky} and \eqref{eq: sigmaky2} yield
$$\|\sum_{n \in F\setminus A} a_ne_n\| \le \beta \overline{K}(1+ |a_{\rho(k)}|\|\sum_{n \in E} e_n\|)$$
where $\beta$ is the basis constant. Hence
\begin{equation} \label{eq: sumoverA}
\|G_k(x)\| = \| x - \sum_{n \in F\setminus A} a_ne_n \| \le 1+  \beta \overline{K}(1+ |a_{\rho(k)}|\|\sum_{n \in E} e_n\|). \end{equation}
Now consider
$$ z:= x - \sum_{n \in A} a_n e_n.$$ Then $\|\sigma^w_{w(B)}(z)\| \le \|x\|=1$, since $w(A) \le w(B)$. Since $(e_n)$ is semi-greedy there exist 
scalars $(\overline{c}_n)$ ($n \in B$) such that
$$\|z - \sum_{n \in B} \overline{c}_n e_n \| \le \overline{K} \|\sigma^w_{w(B)}(z)\| \le \overline{K}.$$
Hence \begin{equation} \label{eq: AandB}
\|\sum_{n \in A} a_ne_n + \sum_{n \in B} \overline{c}_ne_n\| =
 \|x - (z - \sum_{n \in B}\overline{c}_ne_n)\| \le  \|x\| + \overline{K} = 1+ \overline{K}. \end{equation}
Let $\overline{B} := \{n \in B \colon |\overline{c}_n|\ge |a_{\rho(k)}|\}$.
Then
$$ \sum_{n \in A} a_n e_n + \sum_{n \in \overline{B}} \overline{c}_ne_n
= G_s(\sum_{n \in A} a_n e_n + \sum_{n \in B} \overline{c}_ne_n)$$
for some $k \le s \le k+l \le (1 + 2/\alpha)k$. Hence  Proposition~\ref{prop:largekquasigreedy}
and \eqref{eq: AandB} yield
\begin{equation} \label{eq: AandEestimate}
\|\sum_{n \in A} a_n e_n + \sum_{n \in \overline{B}} \overline{c}_ne_n\|
\le L(\alpha/(\alpha+2))(1+\overline{K})\end{equation} On the other
hand, by Lemma~\ref{lem: lowerest} there exists $M < \infty$, depending only on $(e_n)$, such that \begin{equation}
\label{eq: fromremark}
|a_{\rho(k)}| \| \sum_{n \in A \cup \overline{B}} e_n \| \le M\|\sum_{n \in A} a_n e_n + \sum_{n \in \overline{B}} c_n e_n\|.
\end{equation} From \eqref{eq: cardE}, we obtain
$$ |E| \le \frac{2}{\alpha} k \le \frac{2}{\alpha} |A \cup \overline{B}|,$$
and hence, since $(e_n)$ is superdemocratic (with constant $D$, say),
\begin{equation} \label{eq: superdemestimate}\| \sum_{n \in E}e_n \| \le D(\frac{2}{\alpha} +1)\| \sum_{n \in A \cup \overline{B}} e_n\|.
\end{equation}
Finally, combining \eqref{eq: sumoverA}, \eqref{eq: AandEestimate},
 \eqref{eq: fromremark}, and \eqref{eq: superdemestimate} we get $$\|G_k(x)\| \le 1+\beta \overline{K}+ \beta \overline{K}D(\frac{2}{\alpha}+1)M 
L(\frac{\alpha}{\alpha+2})(1+\overline{K}).$$
\end{proof} Combining Corollary~\ref{cor: almostgreedyequivwalmostgreedy} and Theorem~\ref{thm: finitecotypeequiv} we get a partial  answer to Question~\ref{question: semigreedy}.
\begin{corollary} \label{cor: semigreedywsemigreedyequiv} Suppose that $X$ has finite cotype. Then $(e_n)$ is $w$-semi-greedy for some weight $w$ $\Leftrightarrow$ $(e_n)$ is semi-greedy. \end{corollary}
\section{$w$-almost greedy bases when $w\in c_0$}
\begin{proposition}
\label{bounded}Assume $w\in c_0$ and is nonincreasing.  Let $(e_n)_{n=1}^\infty$ be a normalized $D$-$w$-democratic basis for a (real or complex) Banach space $X$.  Then for every $m\in\mathbb{N}$ there exists $N\in\mathbb{N}$ so that if $A\in\mathbb{N}^m$ with $N\leq\min A$ then
\begin{equation}\label{2}\|\sum_{n\in A}e_n\|\leq D.\end{equation}
If furthermore $w\in\ell_1$, we can choose $M\in\mathbb{N}$ such that \eqref{2} holds for all $A\in\mathbb{N}^{<\infty}$ with $M\leq\min A$.\end{proposition}

\begin{proof}Observe that
$$\lim_{k\to\infty}\sum_{i=k}^{k+m-1}w_i=0,$$
so that we can find $N\in\mathbb{N}$ satisfying
$$w(A)=\sum_{i\in A}w_i\leq\sum_{i=N}^{N+m-1}w_i<w_1=w(\{1\})$$
for all $A\in\mathbb{N}^m$ satisfying $N\leq\min A$.  By $D$-$w$-democracy this means
$$\|\sum_{n\in A}e_n\|\leq D\|x_1\|=D.$$
This proves the first part of the proposition.

Next, assume $w\in\ell_1$, and let $M\in\mathbb{N}$ be such that
$$\sum_{i=M}^\infty w_i<w_1$$
and hence $w(A)<w(\{1\})$ for all $A\in\mathbb{N}^{<\infty}$ with $M\leq\min A$.  Then we have \eqref{2} as before.\end{proof}

\begin{remark}\label{nonincreasing}It is easy to see from the proof of Proposition \ref{bounded} that if $w\in c_0$ is not assumed to be nonincreasing, we can still find a constant $D$ such that for any $m\in\mathbb{N}$ there is a positive integer $N\in\mathbb{N}$ satisfying the property that if $k\geq N$ then there exists $A\in\mathbb{N}^m$ with $k\leq A$ and $\|\sum_{n\in A}e_n\|\leq D$.\end{remark}

\begin{corollary}Let $w\in c_0$.  If $(e_n)_{n=1}^\infty$ is a basis for a (real or complex) Banach space which is both $w$-democratic and spreading then it is equivalent to the canonical basis for $c_0$.\end{corollary}

\begin{proof}Assume $(e_n)_{n=1}^\infty$ is both $w$-democratic and spreading.  By Proposition \ref{bounded} and Remark \ref{nonincreasing} there is a constant $C$ such that $$\|\sum_{n\in A}e_n\|\leq C$$ for arbitrarily large $A\in\mathbb{N}^{<\infty}$.

We claim that $(e_n)_{n=1}^\infty$ is unconditional.  Indeed, every conditional spreading basis dominates the summing basis for $c_0$ by \cite[p4]{FOSZ16}.  Let $C'$ be the domination constant.  Then we would have
$$|A|\leq C'\|\sum_{n\in A}e_n\|\leq CC',$$
which is impossible for sufficiently large $A$.  This proves the claim.

Let $U$ be the unconditional constant for $(e_n)_{n=1}^\infty$, and select any $(a_n)_{n=1}^\infty$ with support in $A$.   Then
$$\frac{1}{2\beta M}\sup_{n\in A}|a_n|\leq\|\sum_{n\in A}a_ne_n\|\leq UC\sup_{n\in A}|a_n|,$$
where $M=\sup_{n\in\mathbb{N}}\|e_n\|$.  By the spreading property, this means $(e_n)_{n=1}^\infty$ is equivalent to the canonical basis for $c_0$.\end{proof}

\begin{definition}A basis $(e_n)_{n=1}^\infty$ is {\bf conservative with constant $\boldsymbol{C}$} whenever
$$\|\sum_{n\in A}e_n\|\leq C\|\sum_{n\in B}e_n\|$$
for all $A,B\in\mathbb{N}^{<\infty}$ satisfying $|A|\leq\ |B|$ and $A<B$.\end{definition}

It is known (see \cite[Theorem 3.4]{DKKT}) that a basis is partially-greedy if and only if it is both quasi-greedy and conservative.

\begin{proposition}\label{conservative-to-democratic}Assume $w\in c_0$.  If a normalized basis $(e_n)_{n=1}^\infty$ is both $C$-conservative and $D$-$w$-democratic then it is $2CD\beta$-democratic, where $\beta$ is the basis constant for $(e_n)_{n=1}^\infty$.  Furthermore,
$$\|\sum_{n\in A}e_n\|\leq CD\;\;\;\text{ for all }A\in\mathbb{N}^{<\infty}.$$\end{proposition}

\begin{proof}Let $A,B\in\mathbb{N}^{<\infty}$ with $|A|\leq|B|$.  By Proposition \ref{bounded} and Remark \ref{nonincreasing} we can find $\Gamma\in\mathbb{N}^{<\infty}$ with $|B|\leq |\Gamma|$ and $A\cup B<\Gamma$, and satisfying
$$\|\sum_{n\in\Gamma}e_n\|\leq D.$$
Note that $|A|\leq |\Gamma|$ and $A<\Gamma$ so that by $C$-conservativeness we now have
$$\|\sum_{n\in A}e_n\|\leq C\|\sum_{n\in\Gamma}e_n\|\leq CD.$$
By taking the difference of partial sum projections we also have
$$CD\leq 2CD\beta\|\sum_{n\in B}e_n\|.$$\end{proof}

\begin{theorem}\label{partially-greedy-and-w-democratic}Assume $w\in c_0$.  A quasi-greedy basis $(e_n)_{n=1}^\infty$ for a real Banach space is both conservative and $w$-democratic if and only if it is equivalent to the canonical basis for $c_0$.\end{theorem}

\begin{remark}It is clear that greediness implies neither $w$-almost-greediness nor $w$-democracy when $w\in c_0$, as for instance $\ell_p$ is greedy but not $w$-democratic in this case.  Theorem \ref{partially-greedy-and-w-democratic} shows that $w$-almost-greediness implies partially-greediness (or almost-greediness, or greediness) if and only if we have
$$(e_n)_{n=1}^\infty\text{ is }w\text{-almost greedy}\;\;\;\Leftrightarrow\;\;\;(e_n)_{n=1}^\infty\approx(f_n)_{n=1}^\infty,$$
where $(f_n)_{n=1}^\infty$ denotes the canonical basis of $c_0$.\end{remark}

\begin{proof}[Proof of Theorem \ref{partially-greedy-and-w-democratic}]
The ``if'' part is trivial, so we need only prove the ``only if'' part.  We now assume $(e_n)_{n=1}^\infty$ is quasi-greedy, conservative, and $w$-democratic.  Then by Proposition \ref{conservative-to-democratic}, we can find a constant $C$ with
$$\|\sum_{n\in A}e_n\|\leq C\;\;\;\text{ for all }A\in\mathbb{N}^{<\infty}.$$
Meanwhile, by quasi-greediness together with Proposition \ref{constant-unconditional}, and making $C$ larger if necessary, we have
$$\|\sum_{n\in A}a_ne_n\|\leq 2C\|\sum_{n\in A}e_n\|\cdot\|(a_n)_{n\in A}\|_\infty$$
for any $A\in\mathbb{N}^{<\infty}$.  Also, if $\beta$ is the basis constant for $(e_n)_{n=1}^\infty$, then by looking at the difference of partial sum projections we obtain
$$\|(a_n)_{n\in A}\|_\infty\leq 2\beta\|\sum_{n\in A}a_ne_n\|.$$
Putting all these inequalities together, it follows that $(e_n)_{n=1}^\infty$ is equivalent to $c_0$.\end{proof}

\begin{theorem}\label{uniformly-complemented}Assume $w\in c_0$.  If $(e_n)_{n=1}^\infty$ is a normalized $w$-almost greedy basis for a real Banach space $X$, then for some $1\leq C<\infty$ and every $m\in\mathbb{N}$ there exists $N\in\mathbb{N}$ so that $(e_n)_{n=N+1}^{N+m}\approx_C(f_n)_{n=1}^m$, where $(f_n)_{n=1}^\infty$ is the canonical basis for $c_0$.  Furthermore, we can choose $C$ such that every subsequence of $(e_n)_{n=1}^\infty$ admits a further subsequence which is $C$-equivalent to the canonical basis for $c_0$.\end{theorem}

\begin{proof}
Let $\beta$ be the basis constant for $(e_n)_{n=1}^\infty$.  We showed in the previous section that every $w$-almost greedy basis is both  quasi-greedy and $w$-democratic.  Fix $m\in\mathbb{N}$, and let $N\in\mathbb{N}$ be as in Proposition \ref{bounded}, which we can do by $w$-democracy.  We claim that there exists a constant $C$ such that
$$\frac{1}{2\beta}\sup_{n=N+1,\cdots,N+m}|a_n|\leq \|\sum_{n=N+1}^{N+m}a_ne_n\|\leq C\sup_{n=N+1,\cdots,m+N}|a_n|$$
for all $(a_n)_{n=1}^\infty\in c_0$.  In the above, the first inequality follows from taking the difference of basis projections, and the second inequality for a constant $C$ follows from combining Propositions \ref{constant-unconditional} and \ref{bounded}.   Without loss of generality we may assume $C\geq 2\beta$, completing the proof of the first part of the theorem.

Now let us prove the ``furthermore'' part.  Note that by Remark \ref{subsequence}, every subsequence of $(e_n)_{n=1}^\infty$ admits a further subsequence which is $C$-$w'$-democratic with $w'\in\ell_1$.  Thus, it is enough to show that $(e_n)_{n=M}^\infty\approx_C c_0$ for some $M\in\mathbb{N}$ whenever $w\in\ell_1$.

By Proposition \ref{bounded} we can find $M\in\mathbb{N}$ with
$$\|\sum_{n\in A}e_n\|\leq C$$
for all $A\in\mathbb{N}^{<\infty}$ with $M\leq\min A$.  Thus, as before we can find $C\geq 2\beta$ with
$$\frac{1}{2\beta}\|(a_n)_{n\in A}\|_\infty\leq\|\sum_{n\in A}a_ne_n\|\leq C\|(a_n)_{n\in A}\|_\infty.$$\end{proof}

\begin{corollary}\label{weakly-null}If $w\in c_0$ then every $w$-almost greedy basis of a real Banach space is weakly null.\end{corollary}

\begin{proof}Suppose $(e_n)_{n=1}^\infty$ is not weakly null.  Then there exists $f\in X^*$ such that $f(e_n)\not\to 0$.  Now find a subsequence such that $|f(x_{n_k})|\to\delta>0$.  Then $(x_{n_k})_{k=1}^\infty$ contains no weakly null subsequence.  In particular, it contains no subsequence equivalent to the $c_0$ basis.  By Theorem \ref{uniformly-complemented} it is not $w$-almost greedy.\end{proof}

\end{document}